\documentclass[final,leqno]{siamltex}

\usepackage{amsfonts,amssymb} 
\usepackage{amsmath} 
\usepackage{color} 
\usepackage{graphicx} 
\usepackage{epsfig,mathrsfs} 
\usepackage[bf,SL,BF]{subfigure} 
\usepackage{fancyhdr} 
\usepackage{CJK} 
\usepackage{caption} 
\usepackage{wrapfig} 
\usepackage{cases} 
\usepackage{bm}
\usepackage{algorithm}
\usepackage{algorithmic}
\newtheorem{remark}{Remark}[section] 
\newtheorem{example}{Example}[section] 

\title{ Convergence proof for the multigird method of the nonlocal model
\thanks{This work was supported by NSFC 11271173. }}

\author{Minghua Chen\thanks{Corresponding author. School of Mathematics and Statistics, Gansu Key Laboratory of Applied Mathematics and Complex Systems,
 Lanzhou University, Lanzhou 730000, P.R. China  (Email: chenmh@lzu.edu.cn).}
        \and Weihua Deng\thanks{ School of Mathematics and Statistics, Gansu Key Laboratory of Applied Mathematics and Complex Systems,
 Lanzhou University, Lanzhou 730000, P.R. China (Email: dengwh@lzu.edu.cn). }}

\begin{document}

\maketitle

\begin{abstract}
Recently, nonlocal models attract the wide interests of scientist. They mainly come from two applied scientific fields: peridyanmics and anomalous diffusion. Even though the matrices of the algebraic equation corresponding the nonlocal models are usually Toeplitz (denote $a_0$ as the principal diagonal element, $a_1$ as the trailing diagonal element, etc).  There are still some differences for the models in these two fields. For the model of anomalous diffusion, $a_0/a_1$ is uniformly bounded; most of the time,  $a_0/a_1$ of the model for peridyanmics is unbounded as the stepsize $h$ tends to zero. Based on the uniform boundedness of $a_0/a_1$, the convergence of the two-grid method is well established [Chan, Chang, and Sun,  {\em SIAM J. Sci. Comput.},  {19} (1998), pp. 516--529; Pang and  Sun, {\em J. Comput. Phys.}, {231} (2012),   pp. 693--703; Chen,  Wang, Cheng, and Deng, {\em BIT}, {54} (2014),  pp. 623--647]. This paper provides the detailed proof of the convergence of the two-grid method for the nonlocal model of peridynamics. Some special cases of the full multigrid and the V-cycle multigrid are also discussed. The numerical experiments are performed to verify the convergence.
\end{abstract}



\pagestyle{myheadings}
\thispagestyle{plain}
\markboth{M. H. CHEN AND W. H. DENG}{MULTIGRID METHOD FOR NONLOCAL MODEL}

\section{Introduction}
Ranging from characterizing peridynamics  \cite{Silling:00} to anomalous diffusion \cite{Metzler:00}, the nonlocal models have been builded in more and more scientific fields. The used nonlocal operators include nonlocal diffusion operators \cite{Du:12}, fractional Laplacian operators \cite{Huang:14}, Riesz fractional derivative \cite{Chen:0013,Yang:11},
and the Riesz tempered fractional derivative \cite{Chen:13,Chen:15}.
In the field of anomalous diffusion, the nonlocal operators are derived in both the time and space directions. For the peridynamics, the nonlocal operators are just applied in the space direction. Mathematically, the nonlocal operators corresponding to these two applied fields have close connections, being examined in \cite{DDDGLM:15,Du:12}.
The nonlocal operator mentioned in this paper has a finite
range of nonlocal interactions measured by a horizon parameter $\delta$ \cite{Du:13,Silling:00}.
When $\delta \rightarrow 0$, the nonlocal effect diminishes and the local
or classical partial differential equation (PDE) models are recovered, if the latter are well-defined. For $\delta>0$, compared with classical PDE models, the complexities are introduced by the nonlocal interactions and the matrices of the resulting discrete systems are no longer sparse. 
In a series of recent studies \cite{Tian:13,Tian:14}, the robust discretizations of the nonlocal models have been well deveopled. Based on the fast Toeplitz solver, the direct solution method for the resulting algebraic equation is discussed in \cite{WaTi12}. In this work, we focus on the efficiently iterative solvers, especially, providing the strict convergence proof for the algorithm.

As is well know that the structure and conditioning  of the resulting coefficient matrix of the numerical scheme play a key role for the effectiveness of the linear solver. For the nonlocal models, the associated stiffness matrices tend to be dense, and its condition number depends on both the nonlocal interaction kernel and horizon parameter $\delta$. When $\delta$ is fixed, the condition number is bounded even the stepsize $h$ tends to zero \cite{Aksoylu:14,Zhou:10}. However, if $\delta$ depends on $h$, the condition number will tend to infinity as $h$ becomes smaller and smaller. So, it is interesting/necessary to understand the performance of the linear solver for the different types of the horizon $\delta$. In particular, finding the effective linear solver should be paid much more attention for the case that $\delta$ depends on $h$. This work focuses on the multigrid method (MGM) with uniform convergence rate for various types of the horizon $\delta$.
MGM has often been shown to be the most efficient iterative method for numerically solving the PDEs \cite{Bramble:87,Hackbusch:01}. For the uniform convergence of the V-cycle MGM, one can refer to  \cite{Chang:15,Xu:92,Xu:02,Xu:97} for the second order elliptic operator and \cite{Chen:16} for the block tridiagonal matrix.
For the multilevel matrix algebras like circulant, tau, Hartely, the V-cycle convergence is theoretically obtained by using some special interpolation operators \cite{Arico:07,Arico:04}. For works on the so-called full MGM, i.e., recursive application of the two-grid method (TGM) procedure, see, e.g., \cite{Chan:98,Fiorentino:96,Serra:02}.

For nonlocal models, peridynamics and anomalous diffusion are two of the most successful applied fields. The common feature of the stiffness matrix of the resulting algebraic equation from the models is to have Toeplitz structure. For the Toeplitz matrix, we denote $a_0$ as the principal diagonal element, $a_1$ as the trailing diagonal element, etc. For the stiffness matrix of nonlocal model describing anomalous diffusion, $a_0/a_1$ is bounded; using this attribute, the uniform convergence of the TGM is theoretically obtained \cite{Chan:98,Chen:14,Pang:12}. However, most of the time, $a_0/a_1$ is unbounded for the stiffness matrix of the peridynamic model. So, some new ideas must be introduced for proving the uniform convergence of the TGM for the nonlocal peridynamic model. This paper provides the detailed proof of the uniform convergence of TGM with unbounded $a_0/a_1$.    Furthermore, the special cases of the full MGM and V-cycle MGM are also discussed. The performed numerical experiments show the effectiveness of the MGM.

The outline of this paper is as follows. In the next section, we discuss the recently introduced finite difference discretizations of the nonlocal operator. The MGM algorithms are introduced in Section 3. In Section 4, we study the uniform  convergence estimates of the TGM for the nonlocal model. Convergence of the full MGM and V-cycle MGM in a special case is analyzed in Section 5. To show the effectiveness of the presented schemes, results of numerical experiments are reported in Section 6. Finally, we conclude the paper with some remarks.

\section{Preliminaries: numerical scheme and multigrid method}
Before delivering the detailed convergence proof of the TGM, in this section, we review and discuss the numerical discretization and multigrid method for the nonlocal model (\ref{2.3}).
\subsection{The nonlocal operator and discretization scheme}\label{sec:1}
In this subsection, we introduce the discretization of the nonlocal operator proposed in \cite{Tian:13} and make some discussions on the generating of the matrix elements and treating of the nonhomogeneous boundaries. Let $\Omega$ be a finite bar in $\mathbb{R}$. Without loss of generality, we take $\Omega=(0,b)$,
$b>0$. For  $u=u(x):\Omega\rightarrow \mathbb{R}$,
the  nonlocal operator $\mathcal{L}_\delta$  is defined by \cite{Tian:13},
\begin{equation}\label{2.1}
\begin{split}
\mathcal{L}_\delta u(x)=\int_{B_\delta(x)}(u(y)-u(x))\gamma_\delta(|x-y|)dy ~~\forall x \in \Omega
\end{split}
\end{equation}
with $B_\delta(x)=\{y \in \mathbb{R}: |y-x|<\delta \}$ denoting a neighborhood centered at $x$ of radius $\delta$,
 which is the horizon parameter and  represents the size of nonlocality;
 the symmetric nonlocal kernel $\gamma_\delta(|x-y|)=0$  if $y \not \in B_\delta(x)$.

The operator $\mathcal{L}$ is used in both the time-dependent nonlocal volume-constrained diffusion problem \cite{Du:12}
\begin{equation} \label{2.2}
\left\{ \begin{split}
 u_t - \mathcal{L}_\delta u   &=f_\delta &  ~~{\rm on}  & ~~\Omega,\, t>0,\\
                        u(x,0)&=u_0      &  ~~{\rm on}  & ~~\Omega \cup \Omega_\mathcal{I},\\
                 u&=g        &  ~~{\rm on}  & ~~ \Omega_\mathcal{I},t>0,
 \end{split}
 \right.
\end{equation}
for the function $u=u(x,t)$ and its steady-state counterpart
\begin{equation} \label{2.3}
\left\{ \begin{split}
  - \mathcal{L}_\delta u      &=f_\delta&  ~~{\rm on}  & ~~\Omega,\\
                 u&=g        &  ~~{\rm on}  & ~~ \Omega_\mathcal{I},
 \end{split}
 \right.
\end{equation}
where $u=g $ denotes a  volumetric constraint imposed on a volume $\Omega_\mathcal{I}$ that has
a nonzero volume and is made to be disjoint
 from $\Omega$. For 1D case, we use
$\Omega_\mathcal{I}=(-\delta,0)\cup (b,b+\delta)$.

Let $\gamma_\delta$ be nonnegative and radial, i.e., $\gamma_\delta=\gamma_\delta(|y-x|)\geq 0$.
As in \cite{Tian:13}, we can rewrite (\ref{2.1}) as
\begin{equation}\label{2.4}
\mathcal{L}_\delta u(x)  =\int_{0}^\delta (u(x+s)-2u(x)+u(x-s))\gamma_\delta(s)ds,
\end{equation}
which makes the nonlocal operator  as a {\em continuum difference operator}, or rather an average of finite difference
operators over a continuum scale $(0,\delta)$ \cite{Tian:13}.
Assuming that $u(x)$ is regular enough, from (\ref{2.4}) there exists
\begin{equation*}
\begin{split}
\mathcal{L}_\delta u(x)
                         = C u''(x)  + \mathcal{O} \left(\int_{0}^\delta s^4\gamma_\delta(s)ds\right),
\end{split}
\end{equation*}
where, $C$, is assumed to be positive and independent of $\delta$, i.e.,
\begin{equation*}
0< C=\int_{0}^\delta s^2\gamma_\delta(s)ds < \infty.
\end{equation*}

Now, we introduce and discuss the discretization scheme of (\ref{2.3}). Denote the ratio of the horizon $\delta$ and the mesh size $h$ as
\begin{equation}\label{2.5}
  R=\frac{\delta}{h}>0 ~~{\rm and }~~r=\lfloor R \rfloor,
\end{equation}
which plays an important role in nonlocal diffusion models. Here $\lfloor R \rfloor$ denotes  the greatest integer that is less than or equal to $R$. And we will use $\lceil R \rceil$ to denote the least integer that is greater than or equal to $R$.

 Let $\Omega=(0,b)$ with $\delta<b$
and  the mesh points $x_i=ih$, $h=b/(N+1)$, $i \in \Omega_N=\{-r,\ldots, 0,1,\ldots,N+1,\ldots,N+1+r\}$, where $r$ is defined by (\ref{2.5});  and $u_{i}$ as the numerical approximation of $u(x_i)$ and  $f_{\delta,i}=f_\delta (x_i)$. Denote $I_p=((p-1)h,ph)$ for $1 \leq p \leq r$, and $I_{r+1}=(rh,Rh)=(rh,\delta)$, and the  piecewise linear  basis function is
\begin{equation*}
\phi_p(x)=\left\{ \begin{array}{lll}
 \displaystyle\frac{x-x_{p-1}}{h} &  x \in [x_{p-1}, x_{p}],\\
 \\
 \displaystyle\frac{x_{p+1}-x}{h} & x \in [x_{p}, x_{p+1}]~~~~ {\rm for }~~~~i \in \Omega_N,\\
 \\
 \displaystyle 0 &  {\rm otherwise}.
 \end{array}
 \right.
\end{equation*}

Eq. (\ref{2.1}) can be rewritten as \cite{Tian:13}
\begin{equation}\label{2.6}
\begin{split}
\mathcal{L}_\delta u(x) & =\sum_{p=0}^{r+1}\int_{0}^\delta \frac{u(x+s)-2u(x)+u(x-s)}{s}\phi_p(s)s\gamma_\delta(s)ds,
\end{split}
\end{equation}
and an asymptotically compatible discretization of the nonlocal operator $\mathcal{L}_\delta$ has the following form
\begin{equation}\label{2.7}
\begin{split}
\mathcal{L}_\delta^h u_i =&\sum_{p=1}^{r} \frac{u_{i-p}-2u_i+u_{i+p}}{ph}\int_{I_p\cup I_{p+1}} \phi_p(s)s \gamma_\delta(s)ds\\
                        &+\frac{u_{i-r-1}-2u_i+u_{i+r+1}}{(r+1)h}\int_{I_{r+1}} \phi_{r+1}(s)s\gamma_\delta(s)ds.
\end{split}
\end{equation}
Note that the above integral over $I_{r+1}$  automatically vanishes when $r = R$.

The  discretization of (\ref{2.3}) then has the following form
\begin{equation}\label{2.8}
\begin{split}
-\mathcal{L}_\delta^h u_i &=f_{\delta,i}, ~~i \in \{1,2,\cdots,N\},\\
                       u_i&=g_i,~~ ~~i \in \{-r,\cdots,0\} \cup \{N+1,\cdots,N+r+1\}
\end{split}
\end{equation}
with the following sketch that characterizes different variables:
\begin{equation*}
\begin{split}
 & \big[
\underset{\text{boundary  points}}{\underbrace{ x_{-r}\cdots x_{-1}, x_{0},   }}~
\underset{\text{interface  points}}{\underbrace{ x_1, x_2 \cdots x_r,   }}~
\underset{\text{internal  points}}{\underbrace{  x_{r+1}\cdots x_{N-r},    }}~
\underset{\text{interface  points}}{\underbrace{ x_{N-r+1} \cdots x_N,    }}~
\underset{\text{boundary  points}}{\underbrace{  x_{N+1} \cdots x_{N+r+1} }} \big];\\
 &\,\big[
\underset{\text{boundary  values}}{\underbrace{ g_{-r}\cdots g_{-1},  g_{0},    }}~
\underset{\text{interface  values}}{\underbrace{ u_1, u_2 \cdots u_r,   }}~
\underset{\text{internal  values}}{\underbrace{  u_{r+1}\cdots u_{N-r},    }}~
\underset{\text{interface  values}}{\underbrace{ u_{N-r+1} \cdots u_N,    }}~
\underset{\text{boundary  values}}{\underbrace{  g_{N+1} \cdots g_{N+r+1} }}
\big].
\end{split}
\end{equation*}

For notational convenience, we let
 \begin{equation}\label{2.9}
 \begin{split}
 &U_{\delta}^h=[u_1,u_2,\ldots,u_N]^{\rm T}, ~~F_\delta^h=[f_{\delta,1},f_{\delta,2},\ldots,f_{\delta,N}]^{\rm T};\\  &~~
 F_{\mathcal{V},\delta}^h=[f_{\mathcal{V},1},f_{\mathcal{V},2},\ldots,f_{\mathcal{V},r+1},0,\ldots,0,
 f_{\mathcal{V},{N-r}},\ldots,f_{\mathcal{V},{N}}]^{\rm T}.
 \end{split}
  \end{equation}
Thus, the finite difference scheme (\ref{2.8}) can be recast as
 \begin{equation}\label{2.10}
\begin{split}
&A_{\delta}^h U_{\delta}^h =F_{\delta}^h+F_{\mathcal{V},\delta}^h,
\end{split}
\end{equation}
where the stiffness matrix $A_{\delta}^h=\{a_{i,j}\}_{i,j=1}^{N}$  has a banded structure given by
\begin{equation}\label{2.11}
a_{i_1,j_1}=a_{i_2,j_2} \quad\mbox{for}~ \; |i_1-j_1|=|i_2-j_2|\leq r+1\,,\quad\mbox{and} \quad a_{i,j}=0 \quad\mbox{otherwise}.
\end{equation}
We denote $a_k=a_{i,j}$ with $k=|i-j|$. The auxiliary vector  $F_{\mathcal{V},\delta}^h$  can be determined by  the following matrix form
\begin{equation}\label{2.12}
 \left[ \begin{matrix}
f_{\mathcal{V},1}    \\
f_{\mathcal{V},2}    \\
\vdots                 \\
f_{\mathcal{V},r}      \\
f_{\mathcal{V},r+1}
 \end{matrix}
 \right] = \left[ \begin{matrix}
g_0    &  g_{-1} &   \ddots  &   g_{1-r}  &   g_{-r} \\
0      &  g_{0}  &   \ddots  &   \ddots   &   g_{1-r} \\
\ddots & \ddots  &   \ddots  &   \ddots   &   \ddots   \\
 0     & \ddots  &   \ddots  &   g_0      &   g_{-1}    \\
 0     &   0     &   \ddots  &     0      &   g_0
 \end{matrix}
 \right ]
 \left [ \begin{matrix}
a_{1}   \\
a_{2}    \\
\vdots                                                 \\
a_{r}      \\
a_{r+1}
 \end{matrix}
 \right],
\end{equation}
and
\begin{equation}\label{2.13}
  \left [ \begin{matrix}
f_{\mathcal{V},N}    \\
f_{\mathcal{V},N-1}    \\
\vdots                 \\
f_{\mathcal{V},N-r+1}      \\
f_{\mathcal{V},N-r}
 \end{matrix}
 \right ]= \left [ \begin{matrix}
g_{N+1}&  g_{N+2}&   \ddots  &   g_{N+r}  &   g_{N+r+1} \\
0      &  g_{N+1}&   \ddots  &   \ddots   &   g_{N+r} \\
\ddots &  \ddots &   \ddots  &   \ddots   &   \ddots   \\
 0     & \ddots  &   \ddots  &   g_{N+1}  &   g_{N+2}    \\
 0     &   0     &   \ddots  &     0      &   g_{N+1}
 \end{matrix}
 \right ]
 \left [ \begin{matrix}
a_{1}   \\
a_{2}    \\
\vdots                                                 \\
a_{r}      \\
a_{r+1}
 \end{matrix}
 \right ].
\end{equation}
In the following, 
we focus on the special case where the kernel $\gamma_\delta(s)$ is taken to be a constant, i.e., $\gamma_\delta(s)=3\delta^{-3}$ \cite{Tian:13}. More general kernel types \cite{Aksoylu:14,Du:12,Tian:13}  can be similarly studied.
The entries of the stiffness matrix $A_{\delta}^h$ can be  explicitly documented by

Case 1: $R \leq 1$.
\begin{equation}\label{2.14}
a_{|i-j|}=a_{i,j}=\left\{ \begin{array}{ll}
 \displaystyle\frac{2}{h^{2}},   & j=i,\\
 \\
\displaystyle -\frac{1}{h^{2}},   & |j-i|=1,\\
\\
\displaystyle 0,                                                                       &{\rm otherwise}.
 \end{array}
 \right.
\end{equation}

Case 2: $R >1$. Let $p=|j-i|\geq 1$. Then
\begin{equation}\label{2.15}
a_{|i-j|}=a_{i,j}= \left\{ \begin{array}{ll}
\displaystyle -2\sum_{p=1}^{r+1}a_{i,p},   & j=i,\\
 \\
\displaystyle-\frac{3}{h^{2}R^3},&  p=1:  r-1,\\
\\
\displaystyle-\frac{3r-1 + (R-r)(r^2+rR-2R^2+3r+3R)  }{2h^{2}R^3r}  ,&p =r,\\
\\
\displaystyle-\frac{(R-r)(2R^2-rR-r^2)  }{2h^{2}R^3(r+1)}, &p =r+1,\\
\\
0,                                                                       &{\rm otherwise}.
 \end{array}
 \right.
\end{equation}
If we insert $R \leq 1$ into (\ref{2.15}), which reduces to (\ref{2.14}).

\subsection{Multigrid method}
Let the finest  mesh points $x_i=ih$, $h=b/(N+1)$.
Define the   multiple level of grids \cite{Chen:16,Xu:97}
\begin{equation}\label{3.1}
  \mathcal{M}_m=\left\{x_i^m=\frac{i}{2^m}b, \,i=1: N_m \right\} ~~{\rm with}~~N_m=2^m-1, m=1 : J,
\end{equation}
where  $\mathcal{M}_m$ represents not only the grid with grid spacing $h_m=2^{(J-m)}h$, but also the space of vectors defined on that grid.
The classical   restriction operator $I_m^{m-1}$ and prolongation operator $I_{m-1}^m$  are, respectively, defined by
 \begin{equation}\label{3.2}
\begin{split}
\nu^{m-1}=I_m^{m-1}\nu^m ~~{\rm with}~~ \nu_i^{m-1}=\frac{1}{4}\left(\nu_{2i-1}^{m}+2\nu_{2i}^{m}+\nu_{2i+1}^{m}\right),
~~~i=1:N_{m-1};
\end{split}
\end{equation}
and
 \begin{equation}\label{3.3}
\begin{split}
\nu^{m}=I_{m-1}^{m}\nu^{m-1}~~{\rm with}~~ I_{m-1}^{m} =2\left(I_m^{m-1}\right)^T.
\end{split}
\end{equation}
We use the coarse grid operators defined by the Galerkin approach  \cite[p.\,455]{Saad:03}
\begin{equation}\label{3.4}
  A_{m-1}=I_m^{m-1}A_mI_{m-1}^{m}, ~~~~m=1 : J;
\end{equation}
and for all intermediate $(m,m-1)$ coarse grids we apply the correction operators \cite[p.\,87]{Ruge:87}
\begin{equation}\label{3.5}
  T^m=I_m- I_{m-1}^{m}A_{m-1}^{-1}I_m^{m-1}A_m=I_m- I_{m-1}^{m}P_{m-1}
\end{equation}
with
\begin{equation*}
  P_{m-1}=A_{m-1}^{-1}I_m^{m-1}A_m.
\end{equation*}
We choose    the damped Jacobi iteration matrix by \cite[p.\,9]{Briggs:00}
\begin{equation}\label{3.6}
  K_m=I-S_mA_m~~{\rm with}~~S_{m}:=S_{m,\omega}=\omega D_m^{-1}
\end{equation}
with a weighting  factor $\omega \in (0,1/3]$, and $D_m$ is the diagonal of $A_m$.

A multigrid process can be regarded as defining a sequence of operators $B_m:\mathcal{M}_m\mapsto \mathcal{M}_m$
which is an approximate inverses of $A_m$ in the sense that $||I-B_mA_m||$ is bounded away from one.
We list the following V-cycle multigrid algorithm \cite{Xu:97}: Algorithm \ref{MGM}.
If  $m=2$, the resulting  Algorithm  \ref{MGM} is TGM \cite{Chen:16,Xu:97}.
\begin{algorithm*}
\caption{ V-cycle Multigrid Algorithm: Define $B_1=A_1^{-1}$. Assume that $B_{m-1}:\mathcal{M}_{m-1}\mapsto \mathcal{M}_{m-1}$ is defined.
We shall now define $B_m:\mathcal{M}_{m}\mapsto \mathcal{M}_{m}$ as an approximate iterative solver for the equation associated with $A_m\nu^m=f_m$.}
\label{MGM}
\begin{algorithmic}[1]
\STATE Pre-smooth: Let $S_{m,\omega}$ be defined by (\ref{3.6}) and   $\nu^m_0=0$, $l=1: m_1$
  $$\nu^m_l=\nu^m_{l-1}+S_{m,\omega_{pre}}(f_m-A_m\nu^m_{l-1}).$$
\STATE Coarse grid correction: $e^{m-1} \in \mathcal{M}_{m-1}$ is the approximate solution of the residual equation $A_{m-1}e=I_m^{m-1}(f_m-A_m\nu^m_{m_1})$
by the iterator $B_{m-1}$:
$$e^{m-1}=B_{m-1}I_m^{m-1}(f_m-A_m\nu^m_{m_1}).$$
\STATE Post-smooth:~~$\nu^m_{m_1+1}=\nu^m_{m_1}+I_{m-1}^{m}e^{m-1}$  and $l=m_1+2: m_1+ m_2$
$$\nu^m_l=\nu^m_{l-1}+S_{m,\omega_{post}}(f_m-A_m\nu^m_{l-1}).$$
\STATE Define $B_mf_m=\nu^m_{m_1+m_2}$.
\end{algorithmic}
\end{algorithm*}

\section{Convergence of TGM for nonlocal model}
Now, we start to prove the convergence of the TGM for nonlocal model. First, we give some Lemmas that will be used.
\begin{lemma}\cite[p.\,5]{Chan:07} \label{lemmma4.1}
Given  $n\times n$ symmetric matrices $P$ and $Q$  and let $P'$ be a  principal submatrix of $P$ of order $n-1$.
  Then, for $m=1,2,\ldots,n$,
\begin{eqnarray}
\label{4.1}
 && \lambda_k(P)+ \lambda_1(Q-P) \leq \lambda_k(Q) \leq \lambda_k(P) + \lambda_n(Q-P),\\
 \label{4.2}
&& \lambda_1(P)\leq \lambda_1(P')\leq \lambda_2 (P) \leq \lambda_2 (P') \leq \cdots \leq\lambda_{n-1}(P')\leq\lambda_n(P),\\
&&
\label{4.3}
\lambda_{\min}(P)=\lambda_1(P)=\min_{x\neq 0}\frac{x^TPx}{x^Tx},~~ \lambda_{\max}(P)=\lambda_n(P)=\max_{x\neq 0}\frac{x^TPx}{x^Tx}.
\end{eqnarray}
\end{lemma}
\begin{definition}\cite[p.\,13]{Chan:07}\label{definition4.2}
Let  $n \times n$ Toeplitz  matrix  $T_n$ be of the following form:
\begin{equation*}
T_n=\left [ \begin{matrix}
                      t_0           &      t_{-1}             &      \cdots         &       t_{2-n}       &       t_{1-n}      \\
                      t_{1}         &      t_{0}              &      t_{-1}         &      \cdots         &       t_{2-n}        \\
                     \vdots         &      t_{1}              &      t_{0}          &      \ddots         &        \vdots            \\
                     t_{n-2}        &      \cdots             &      \ddots         &      \ddots         &        t_{-1}    \\
                     t_{n-1}        &       t_{n-2}           &      \cdots         &       t_1           &        t_0
 \end{matrix}
 \right ];
\end{equation*}
i.e., $t_{i,j}=t_{i-j}$ and $T_n$ is constant along its diagonals. Assume that the diagonals $\{t_k\}_{k=-n+1}^{n-1}$ are the Fourier coefficients of a function
$f$, i.e.,
\begin{equation*}
  t_k=\frac{1}{2\pi}\int_{-\pi}^{\pi}f(x)e^{-ikx}dx.
\end{equation*}
Then the function $f$ is called the generating function of $T_n$.
\end{definition}

\begin{lemma}\cite[p.\,13-15]{Chan:07}\label{lemma4.3} (Grenander-Szeg\"{o} theorem) Let $T_n$ be given by above matrix with a generating function $f$,
where $f$ is a $2\pi$-periodic continuous real-valued functions defined on $[-\pi,\pi]$.
Let $\lambda_{\min}(T_n)$ and $\lambda_{\max}(T_n)$ denote the smallest and largest eigenvalues of $T_n$, respectively. Then we have
\begin{equation*}
  f_{\min} \leq \lambda_{\min}(T_n) \leq \lambda_{\max}(T_n) \leq f_{\max},
\end{equation*}
where $f_{\min}$ and  $f_{\max}$  is the minimum and maximum values of $f(x)$, respectively.
Moreover, if $f_{\min}< f_{\max}$, then all eigenvalues of $T_n$ satisfy
\begin{equation*}
  f_{\min} < \lambda(T_n) < f_{\max},
\end{equation*}
for all $n>0$. In particular, if $f_{\min}>0$, then $T_n$ is positive definite.
\end{lemma}

\begin{lemma}\label{lemma4.4}
Let the discrete Laplacian-like operators  $\{L_j\}_{j=1}^{N-1}$ be defined by
\begin{equation}\label{4.4}
L_{j}=\left [ \begin{matrix}
2      &\overset{j-1 \text{  zeros }}{\overbrace{  \cdots }} &   -1       &           &             \\
\ddots &      \ddots                                                    &  \ddots    &  \ddots   &              \\
 -1    &     \ddots                                                &  \ddots    &  \ddots   &      -1       \\
       &             \ddots                                        &  \ddots    &   \ddots  &     \ddots     \\
       &                                                           &    -1      &   \ddots  &      2
 \end{matrix}
 \right ]_{N \times N}~~{ with}~~\quad  1\leq j \leq N-1.
\end{equation}
Then, the smallest eigenvalues of $L_j$ satisfy
\begin{equation*}
  \lambda_1(L_j)\geq 4\sin^2\left(\frac{\pi  }{2\left(\lceil N/j\rceil +1\right)}\right),~~j=1,2,\ldots N-1.
\end{equation*}
Moreover, if $N/j$ is an integer
\begin{equation*}
  \lambda_k(L_j)= 4\sin^2\left(\frac{k\pi }{2(N/j +1)} \right), ~~k=1,2,\ldots N/j.
\end{equation*}

\end{lemma}
\begin{proof}
Let $\nu_1^{k,1}=[\nu_1^k,\nu_2^k,\ldots,\nu_N^k]^T$ be the associated eigenvector with the tridiagonal  matrix $L_1$.
It is well known that its eigenvalues are given by \cite[p.\,702]{Stoer:02}
\begin{equation*}
  \lambda_{k,1}=4\sin^2\left(\frac{k\pi }{2(N+1)} \right), ~k=1,2,\ldots  N.
\end{equation*}
Define
$$\nu_i^{k,j}=[\underset{ {i-1}}{\underbrace{  0 \cdots 0 }},\nu_1^k,\underset{ {j-i}}
{\underbrace{  0 \cdots 0 }},\underset{ {i-1}}{\underbrace{  0 \cdots 0 }},\nu_2^k,\underset{ {j-i}}
{\underbrace{  0 \cdots 0 }},\ldots,\underset{ {i-1}}{\underbrace{  0 \cdots 0 }},\nu_N^k,\underset{ {j-i}}{\underbrace{  0 \cdots 0 }}]^T,~~i=1,2,\cdots,j.$$
Then, for the  matrix $L_j$ with dimension ${(jN)}\times{(jN)}$, we have
\begin{equation*}
   \{L_j \}_{{(jN)}\times{(jN)}}\,\nu_i^{k,j}=\lambda_{k,1} \nu_i^{k,j}=4\sin^2\left(\frac{k\pi }{2(N+1)} \right)\nu_i^{k,j}, ~~ i=1,2,\cdots,j,
\end{equation*}
leading to all eigenvalues with multiplicity $j$ and eigenvectors of  $\{L_j\}_{{(jN)}\times{(jN)}}$.
A dimension rescaling then shows that
\begin{equation}\label{4.5}
   \lambda_k(\{L_j\}_{{N}\times{N}}):=\lambda_{k,j}=4\sin^2\left(\frac{k\pi }{2(N/j +1)}\right), ~k=1,2,\ldots  N/j
\end{equation}
if $N/j$ is an integer.

If  $N/j$ is not an integer,  we extend  $N$ to $\widetilde{N}$ such that $\widetilde{N}/j$ is an integer, i.e.,
 $$\widetilde{N}/j:=\lceil N/j\rceil=\frac{N+j-\!\!\!\!\!\!\mod(N,j)}{j}\, ,$$
where $\mod(N,j)$ means the remainder of division of $N$ by $j$.

From (\ref{4.5}) and (\ref{4.2}), we obtain
\begin{equation*}
  \lambda_{1}(\{L_j\}_{\widetilde{N}\times\widetilde{N}})=4\sin^2\left(\frac{\pi }{2(\widetilde{N}/j +1)}\right)
  =4\sin^2\left(\frac{\pi  }{2\left(\lceil N/j\rceil +1\right)}\right),
\end{equation*}
and
\begin{equation*}
  \lambda_{1}(\{L_j\}_{{N}\times{N}}) \geq \lambda_{1}(\{L_j\}_{\widetilde{N}\times\widetilde{N}})
  =4\sin^2\left(\frac{\pi  }{2\left(\lceil N/j\rceil +1\right)}\right).
\end{equation*}
The proof is completed.
\end{proof}

\begin{lemma}\label{lemma4.5}
Let the matrix  $A_\delta^h$ be defined by (\ref{2.11}) and (\ref{2.15}) on a finite bar $\Omega=(0,b)$, $b>0$.
Let  $\delta=ch^\beta$,  $\beta\geq 0$, $h\rightarrow 0$ and $c>0$. Then
\begin{equation*}
  \lambda_{\min}(A_\delta^h)\geq \frac{1}{27b^2}.
\end{equation*}
\end{lemma}
\begin{proof}
According the definition of $L_j$ given in Lemma \ref{lemma4.4}, we can recast (\ref{2.11}) with its elements defined by (\ref{2.15}) as
\begin{equation}\label{4.6}
  A_\delta^h=-a_1L_1-a_2L_2\cdots -a_{r+1}L_{r+1}=-\sum_{j=1}^{r+1}a_jL_j,
\end{equation}
where $\{a_j\}$ are entries on different off-diagonals.
We should  check  the following two  cases: $r\leq 1$ and $r\geq 2$.

Case 1: $r\leq 1$.  From (\ref{2.14}) and (\ref{4.6}), we obtain
$$A_\delta^h=-a_1L_1,$$ which means that
$$\lambda_{\min}(A_\delta^h)=-a_1\lambda_{\min}(L_{1})
=\frac{1}{h^2}\cdot4\sin^2\left(\frac{\pi  }{2\left( N +1\right)}\right)
=\frac{1}{h^2}\cdot4\sin^2\left(\frac{\pi  h}{2b}\right)=\frac{\pi^2}{b^2}+\mathcal{O} (h^2).$$

Case 2: $r\geq 2$. Using
$$\frac{2}{\pi}x \leq \sin(x) \leq x,~ x\in\left[0,\frac{\pi}{2}\right],$$
we obtain
\begin{equation*}
\begin{split}
\lambda_{\min}(A_\delta^h)
&\geq -\sum_{j=1}^{r+1}a_j\lambda_{\min}(L_j)
\geq  \frac{1}{h^2}\frac{3}{R^3}\sum_{j=1}^{r-1}4\sin^2\left(\frac{\pi  }{2\left(\lceil N/j\rceil +1\right)}\right)\\
&\geq  \frac{1}{h^2}\frac{3}{R^3}\sum_{j=1}^{r-1}4\sin^2\left(\frac{\pi  }{2\left( N/j +2\right)}\right)
\geq  \frac{1}{h^2}\frac{3}{R^3}\sum_{j=1}^{r-1}4\sin^2\left(\frac{j\pi  }{6N}\right)\\
&\geq  \frac{1}{h^2}\frac{3}{R^3}\sum_{j=1}^{r-1}4\sin^2\left(\frac{j\pi h }{6b}\right)
\geq  \frac{4}{3b^2R^3}\sum_{j=1}^{r-1}j^2\\
&\geq \frac{4}{3b^2(r+1)^3}\cdot\frac{1}{6}\cdot\frac{r+1}{3}\cdot\frac{r+1}{2}\cdot\frac{2(r+1)}{2}=\frac{1}{27b^2}.
\end{split}
\end{equation*}
The proof is completed.
\end{proof}
\begin{remark}
From (\ref{4.6}), we  know that the discrete nonlocal operator $A_\delta^h$ can be viewed as  the superposition of discrete Laplacian-like operators  $\{L_j\}_{j=1}^{r+1}$; and it  reduces to the classical discrete Laplacian operator  when $r=0$ or $1$.
\end{remark}

\begin{lemma}\label{lemma4.6}
Let the matrix  $A_\delta^h$ be defined by  (\ref{2.11}) and (\ref{2.15}) on a finite bar $\Omega=(0,b)$, $b>0$.
Let  $\delta=ch^\beta$,  $\beta\geq 0$, $h\rightarrow 0$ and $c>0$. Then,
 there exists the bound of the condition number
$${\rm cond}\,(A_\delta^h)=\frac{\lambda_{\max}(A_\delta^h)}{\lambda_{\min}(A_\delta^h)} \leq c_* \min\{\delta^{-2},h^{-2}\},$$
where $c_*$ is a positive constant. 
\end{lemma}
\begin{proof}
Case 1: $\beta> 1.$ From (\ref{2.5}), there exists
\begin{equation*}
  R=\frac{\delta}{h}=ch^{\beta-1}\leq c_0  ~~{\rm with}~~c_0  ~~{\rm a ~~constant}.
\end{equation*}
If $ R \leq 1$,  from (\ref{2.14}) and (\ref{4.6}), we obtain
$$A_\delta^h=-a_1L_1,$$
which means that
$$\lambda_{\max}(A_\delta^h)=-a_1\lambda_{\max}(L_{1})
=\frac{1}{h^2}\cdot4\sin^2\left(\frac{N\pi  }{2\left( N +1\right)}\right)
\leq \frac{4}{h^2}.$$
If $ 1< R \leq c_0$,  according to (\ref{4.6}), (\ref{4.1}),  (\ref{4.4}) and  (\ref{2.15}),  there exists
\begin{equation*}
\lambda_{\max}(A_\delta^h) \leq -\sum_{j=1}^{r+1}a_j\lambda_{\max}(L_j)
                  \leq  \frac{1}{h^2}\frac{3}{R^3}\sum_{j=1}^{r+1}4
                  \leq \frac{12(R+1)}{R^3}\frac{1}{h^2}\leq \frac{12(c_0+1)}{h^2}.
\end{equation*}

Case 2: $0\leq \beta \leq  1.$ Using  (\ref{2.15}), we obtain
  \begin{equation}\label{4.7}
\begin{split}
  a_0&=-2\sum_{m=1}^{r+1}a_{i,m} \leq   2\sum_{m=1}^{r+1}\frac{3}{h^{2}R^3} \leq \frac{12}{\delta^2}.
\end{split}
\end{equation}
From the Gerschgorin theorem \cite[p.\,133]{Thomas:95},   the eigenvalues $\lambda$ of the matrix  $A_\delta^h$ satisfy
\begin{equation*}
  \lambda_{\max}(A_\delta^h) \leq 2a_0 \leq \frac{24}{\delta^2}.
\end{equation*}
From Lemma \ref{lemma4.5}, and the discussions of Case 1 and Case 2, the desired result is obtained.
\end{proof}

Since the matrix $A^h_\delta$ is symmetric positive definite, we can define the following inner products
\begin{equation*}
  ( u,v  )_D=(Du,v), \quad (u,v)_{A}=(Au,v), \quad (u,v)_{AD^{-1}A}=(Au,Av)_{D^{-1}},
\end{equation*}
where for convenience we have dropped the explicit notational dependence on $h$ and $\delta$
so that $A:=A_J=A^h_\delta$ and $D$ is its diagonal. Here $(\cdot,\cdot)$ is the usual Euclidean inner product.
\begin{lemma} \cite[p.\,84]{Ruge:87}\label{lemma4.7}
Let $A_J$ be a symmetric positive definite. If $\eta\leq\omega(2-\omega \eta_0)$ with $\eta_0\geq\lambda_{\max}(D_J^{-1}A_J) $,  then the damped Jacobi iteration
with relaxation parameter $0 < \omega < 2/\eta_0 $ satisfies
\begin{equation}\label{4.8}
 ||K_J\nu^J||_{A_J }^2 \leq  ||\nu^J||_{A_J }^2 - \eta ||A_J \nu^J||_{D_J ^{-1}}^2 \quad  \forall \nu^J \in \mathcal{M}_J.
\end{equation}
\end{lemma}

\begin{lemma}\cite[p.\,89]{Ruge:87}\label{lemma4.8}
Let $A_J$ be a symmetric positive definite matrix and $K_J$ satisfies (\ref{4.8}) and
\begin{equation}\label{4.9}
   \min_{\nu^{J-1} \in \mathcal{M}_{J-1} }||\nu^{J}-I_{J-1}^J\nu^{J-1}||_{D_J}^2\leq \kappa ||\nu^J||_{A_J}^2 \quad  \forall \nu^J \in \mathcal{M}_J
\end{equation}
with  $\kappa>0$ independent of $\nu^J$. Then, $\kappa\geq \eta>0$ and the convergence factor of TGM  satisfies
\begin{equation*}
||K_JT^J||_{A_J}\leq \sqrt{1-\eta/\kappa }\quad \forall \nu^J\in \mathcal{M}_J.
\end{equation*}
\end{lemma}

We present the convergence results of TGM in Theorem \ref{theorema4.9} and Theorem \ref{theorema4.11}. The first part of the proof of Theorem \ref{theorema4.9} follows the traditional idea \cite{Chan:98,Chen:14,Pang:12}, but just the convergence result for $\beta\ge 1$ is obtained; we use a different technique to prove the case $\beta=0$. After using the new idea, the convergence results for $\beta \ge 0$ are got, being proposed in Theorem \ref{theorema4.11} .

\begin{theorem}\label{theorema4.9}
Let $A_J=A^h_\delta$ be defined by (\ref{2.10}) and (\ref{2.15}) on a finite bar $\Omega \in(0,b)$, where  $\delta=ch^\beta$,  $\beta\geq 0$, $h\rightarrow 0$ and $c>0$.
Then $K_{J}$ satisfies (\ref{4.8}) and the convergence factor of the TGM satisfies
\begin{equation*}
||K_{J} T^J||_{A_J} \leq \left\{ \begin{split}
&\sqrt{1-\eta/c_0 }<1~~~~~~~~{\rm with }~~\beta\geq 1, c_0=\max(1,2c),\\
&\sqrt{1-\eta c^2/(648 b^2)}<1~~{\rm with }~~\beta=0,
 \end{split}
 \right.
\end{equation*}
where $\eta\leq2\omega(1-\omega)$ with $0 < \omega <1$.
\end{theorem}
\begin{proof}
Since $\lambda_{\max}(D_J^{-1}A_J)\leq \eta_0$ with $\eta_0=2$, it leads to $0 < \omega < 2/\eta_0=1$.
From Lemma \ref{lemma4.7}, we conclude that $K_J$ satisfies (\ref{4.8}) with  $\eta\leq2\omega(1-\omega)$. Let
$$\nu^J=(\nu_1,\nu_2,\ldots,\nu_{N})^{\rm T} \in \mathcal{M}_J, ~~ \nu^{J-1}=(\nu_2,\nu_4,\ldots,\nu_{N-1})^{\rm T} \in \mathcal{M}_{J-1},$$
and $\nu_0=\nu_{N+1}=0$ with $N=2^J-1$  in (\ref{3.1}).
From \cite{Chan:98,Chen:14,Pang:12} and (\ref{4.6}), we have
\begin{equation}\label{4.10}
\begin{split}
&||\nu^J-I_{J-1}^J\nu^{J-1}||_{D_J}^2\leq  a_0 \sum_{i=1}^{N}\left(\nu_{i}^2 -\nu_{i}\nu_{i+1}\right),
\end{split}
\end{equation}
\begin{equation}\label{4.11}
\begin{split}
&\sum_{i=1}^{N}\nu_{i}^2 \geq \sum_{i=1}^{N} \left|\nu_{i}\nu_{i+1}\right|,
\end{split}
\end{equation}
and
\begin{equation}\label{4.12}
\begin{split}
&||\nu^J||_{A_J}^2=(\nu^J,A_J\nu^J) \geq  (\nu^J,-a_1L_{1}\nu^J)=-2a_1\sum_{i=1}^{N}\left( \nu_i^2-\nu_i\nu_{i+1}   \right),
\end{split}
\end{equation}
where $a_0$ and $a_1$ are given in (\ref{2.15}). According to (\ref{4.10}) and (\ref{4.12}), there exists
\begin{equation}\label{4.13}
  ||\nu^J-I_{J-1}^J\nu^{J-1}||_{D_J}^2 \leq  \frac{a_0}{-2a_1} ||\nu^J||_{A_J}^2.
\end{equation}
Next we prove that (\ref{4.9}) holds.

Case 1: $\beta\geq 1.$ Using  (\ref{2.6}), there exists
\begin{equation*}
  R=\frac{\delta}{h}=ch^{\beta-1}\leq c  ~~{\rm as}~~h\leq 1.
\end{equation*}
When $ R \leq 1$, from (\ref{2.14}), it leads to $\kappa=\frac{a_0}{-2a_1}=1$; then from Lemma \ref{lemma4.8},
\begin{equation}\label{4.14}
||K_J T^J||_{A_J} \leq \sqrt{1-\eta }.
\end{equation}
When $ 1< R \leq c$, using  (\ref{2.15}) and (\ref{4.7}), we have
\begin{equation}\label{4.15}
\begin{split}
  a_1&=-\frac{1}{h^{2}}\cdot\frac{3}{R^3}=-\frac{3}{\delta^3}h=-\frac{3}{c^3}h^{1-3\beta},\\
  a_0&=-2\sum_{m=1}^{r+1}a_{i,m} \leq \frac{12}{\delta^2}=\frac{12}{c^2}h^{-2\beta}.
\end{split}
\end{equation}
It leads to
\begin{equation}\label{4.16}
  \kappa=\frac{a_0}{-2a_1}\leq 2ch^{\beta-1},
\end{equation}
i.e.,   $\kappa \leq 2c ~~{\rm as}~~h\leq 1$.  Thus we obtain
\begin{equation}\label{4.17}
  ||K_{J} T^J||_{A_J} \leq \sqrt{1-\eta/(2c) }<1.
\end{equation}
Combining (\ref{4.14}) and (\ref{4.17}), it yields
$$||K_{J} T^J||_{A_J} \leq \sqrt{1-\eta/c_0 }~~{\rm with }~~\beta\geq 1, \,c_0=\max(1,2c).$$

Case 2: $\beta=0$, i.e., $\delta=c$. Since $ \kappa \rightarrow \infty$ as  $\beta=0$ in the estimate  (\ref{4.16}),  we need to look for an estimate of the other form.  From  (\ref{4.15}), we obtain  $a_0\leq \frac{12}{c^2}.$
Using (\ref{4.11}), it yields
\begin{equation}\label{4.18}
  ||\nu^J||^2=\sum_{i=1}^{N}\nu_{i}^2=\frac{1}{2}\sum_{i=1}^{N}\nu_{i}^2+\frac{1}{2}\sum_{i=1}^{N}\nu_{i}^2
                                \geq \frac{1}{2}\sum_{i=1}^{N}\left(\nu_{i}^2 -\nu_{i}\nu_{i+1}\right).
\end{equation}
From  Lemma \ref{lemma4.5}, we have
\begin{equation}\label{4.19}
||\nu^J||_{A_J}^2=(A_J\nu^J,\nu^J) \geq \lambda_{\min}(A_J)||\nu^J||^2 \geq \frac{1}{27b^2}||\nu^J||^2.
\end{equation}
According to (\ref{4.10}), (\ref{4.18}), (\ref{4.19})  and  (\ref{4.15}), we get
\begin{equation*}
\begin{split}
 ||\nu^J-I_{J-1}^J\nu^{J-1}||_{D_J}^2
&\leq  a_0 \sum_{i=1}^{N}\left(\nu_{i}^2 -\nu_{i}\nu_{i+1}\right)\leq 2 a_0 ||\nu^J||^2
\leq 54b^2a_0 ||\nu^J||_{A_J}^2\leq \kappa ||\nu^J||_{A_J}^2
\end{split}
\end{equation*}
with
$\kappa=54b^2a_0\leq \frac{ 648b^2}{c^2}.$
Hence  $$||K_{J} T^J||_{A_J} \leq \sqrt{1-\eta c^2/(648 b^2)}~~{\rm with }~~\beta= 0.$$
The proof is completed.
\end{proof}

In the works \cite{Chan:98,Chen:14,Pang:12}, the convergence factor of the two-grid method is uniformly bounded below one independent of $h$ by estimating $\kappa=\frac{a_0}{2|a_1|}<\infty$, $a_1 \neq 0$.
Since $ \kappa=\frac{a_0}{2|a_1|} \rightarrow \infty$ as  $\beta \in [0,1)$ in the estimate (\ref{4.16}), next we need to use a different idea to prove the case: $\beta \geq 0$.
\begin{lemma}\label{lemma4.10}
Let $A=\sum\limits_{j=1}^nL_{j}$ and $B=nL_{1}$ with $n\geq 1$, where  $L_{j}$ are defined by  (\ref{4.4}).  Then $2A-B$ is a  positive definite matrix.
\end{lemma}
\begin{proof}
The generating functions of $A$ and $B$ are
\begin{equation*}
  f_{A}(x)=2n-2\sum_{k=1}^{n}\cos(kx)~~\mbox{and}~~ f_{B}(x)=2n(1-\cos x),
\end{equation*}
respectively.
Since $f_{A}(x)$ and $f_B(x)$ are the even function and  $2\pi$-periodic continuous real-valued functions defined on $[-\pi,\pi]$, we just need to consider on $[0,\pi]$. Moreover
\begin{equation}\label{4.20}
  2f_{A}(x)-f_B(x)=4ng(x)
\end{equation}
with
\begin{equation}\label{4.21}
g(x)= \cos^2\frac{x}{2}-\frac{1}{n}\sum\limits_{k=1}^{n}\cos(kx),~~~~ x\in[0,\pi].
\end{equation}

Next we  prove $g(x)\geq 0.$ If $x=0$, it yields  $g(x)=0$. Denote
\begin{equation}\label{4.22}
\varphi_n(x):=\frac{1}{n}\sum\limits_{k=1}^{n}\cos(kx)=\frac{\sin\frac{2n+1}{2}x-\sin\frac{x}{2}}{2n\sin\frac{x}{2}},~~~~ x\in(0,\pi].
\end{equation}

Case 1: $0<\frac{2n+1}{2}x \leq \pi$. We can rewrite (\ref{4.22}) as
\begin{equation*}
\varphi_n(x)=\frac{x}{2\sin\frac{x}{2}}    \phi (y)  ~~~~{\rm with} ~~~~ \phi (y)=\frac{\sin y-\sin\frac{x}{2}}{y-\frac{x}{2}},~~
y=\frac{2n+1}{2}x,~~~x \in \left(\frac{x}{2},\pi\right].
\end{equation*}
It is easy to prove that $\phi (y)$ decreases with respect to $y$, which implies
\begin{equation*}
\varphi_n(x) \leq \varphi_{n-1}(x)\leq \cdots \leq \varphi_1(x)=\frac{\sin\frac{3}{2}x-\sin\frac{x}{2}}{2\sin\frac{x}{2}}=\cos x< \cos^2\frac{x}{2},
\end{equation*}
i.e., $g(x)>0$.

Case 2: $\pi\leq \frac{2n+1}{2}x \leq 2\pi+\frac{x}{2}$. Since $\varphi_n(x)\leq 0 <\cos^2\frac{x}{2}$, it yields  $g(x)>0$.

Case 3: $ \frac{2n+1}{2}x \geq 2\pi+\frac{x}{2}$. Using (\ref{4.22}), there exists
\begin{equation*}
\varphi_n(x)\leq \frac{1-\sin\frac{x}{2}}{2n\sin\frac{x}{2}}=\frac{\cos^2\frac{x}{2}}{2n\sin\frac{x}{2}\left(1+\sin\frac{x}{2}\right)}<\cos^2\frac{x}{2},
\end{equation*}
since
$$2n\sin\frac{x}{2}\left(1+\sin\frac{x}{2}\right)\geq \frac{4\pi}{x}\sin\frac{x}{2}\geq \frac{4\pi}{x}\cdot\frac{2}{\pi}\cdot\frac{x}{2}=4.$$
According to the above equations and Lemma \ref{lemma4.3},
the desired result is obtained.
\end{proof}

\begin{theorem}\label{theorema4.11}
Let $A_J=A^h_\delta$ be defined by (\ref{2.10}) and (\ref{2.15}) on a finite bar $\Omega \in(0,b)$, where  $\delta=ch^\beta$,  $\beta\geq 0$, $h\rightarrow 0$ and $c>0$.
Then $K_{J}$ satisfies (\ref{4.8}) and the convergence factor of the TGM satisfies
\begin{equation*}
||K_{J} T^J||_{A_J} \leq \sqrt{1-\eta/6 }<1,~~~~\beta\geq 0,
\end{equation*}
where $\eta\leq2\omega(1-\omega)$ with $0 < \omega <1$.
\end{theorem}
\begin{proof}
By the proof of  Theorem \ref{theorema4.9},
we know that $K_J$ satisfies (\ref{4.8}) and
\begin{equation}\label{4.23}
\begin{split}
&||\nu^J-I_{J-1}^J\nu^{J-1}||^2\leq   \sum_{i=1}^{N}\left(\nu_{i}^2 -\nu_{i}\nu_{i+1}\right)=\frac{1}{2}\left(L_1\nu^J,\nu^J  \right).
\end{split}
\end{equation}

Case 1: $r\leq 1$.  From (\ref{2.14}) and (\ref{4.6}), we obtain
$$A_J=A_\delta^h=-a_1L_1,$$ which means that
\begin{equation}\label{4.24}
||\nu^J||_{A_J}^2=(A_J\nu^J,\nu^J) =(-a_1L_1\nu^J,\nu^J).
\end{equation}
 According to (\ref{4.23}) and (\ref{4.24}), there exists
\begin{equation*}
  ||\nu^J-I_{J-1}^J\nu^{J-1}||_{D_J}^2 \le \frac{a_0}{2}\left(L_1\nu^J,\nu^J  \right)=\frac{a_0}{-2a_1}||\nu^J||_{A_J}^2=||\nu^J||_{A_J}^2.
\end{equation*}
 Thus from Lemma \ref{lemma4.8}, we obtain
\begin{equation*}
  ||K_{J} T^J||_{A_J} \leq \sqrt{1-\eta }~~~{\rm with}~~~r\leq 1.
\end{equation*}

Case 2: $r\geq  2$. According to (\ref{4.6}), (\ref{2.15}), Lemma \ref{lemma4.10} and (\ref{4.23}), we obtain
\begin{equation*}
\begin{split}
||\nu^J||_{A_J}^2
&=\left( A_J\nu^J,\nu^J \right) \geq  -a_1\left( \sum_{j=1}^{r-1}L_j\nu^J,\nu^J \right)\geq - \frac{a_1(r-1)}{2}\left( L_1\nu^J,\nu^J \right)\\
&\geq - \frac{a_1(r+1)}{6}\left( L_1\nu^J,\nu^J \right)\geq  \frac{a_0}{12}\left( L_1\nu^J,\nu^J \right)
\geq  \frac{a_0}{6}||\nu^J-I_{J-1}^J\nu^{J-1}||^2\\
&=\frac{1}{6}||\nu^J-I_{J-1}^J\nu^{J-1}||_{D_J}^2.
\end{split}
\end{equation*}
Thus from Lemma \ref{lemma4.8}, we have
\begin{equation*}
  ||K_{J} T^J||_{A_J} \leq \sqrt{1-\eta/6 }~~~{\rm with}~~~r\geq 2.
\end{equation*}
The proof is completed.
\end{proof}

\section{Convergence of the full MGM and V-cycle MGM with $\bf{\delta=ch}$}
We extend the convergence results of TGM given in the above section to the full MGM and V-cycle MGM in case that $\delta=ch$ with $c$ being an appropriate natural number. First, we introduce some lemmas. We will use the notion of M-matrix, which is a positive definite matrix with positive entries on the diagonal and nonpositive off-diagonal entries. And another notion is called weakly diagonal dominant \cite[p.\,3]{Briggs:00}, if the diagonal element of a matrix is at least as large as the sum of the off-diagonal elements in the same row or column.

\begin{lemma}[\cite{Chen:16}] \label{lemma5.1}
Let $A^{(1)}=\{a_{i,j}^{(1)}\}_{i,j=1}^{\infty}$ with $a_{i,j}^{(1)}=a_{|i-j|}^{(1)}$ be a  symmetric  Toeplitz  matrix
and  $A^{(k)}=L_h^{H}A^{(k-1)}L_{H}^{h}$ with $L_h^{H}=4I_k^{k-1}$ and $L_H^{h}=(L_h^{H})^T$. Then
$A^{(k)}$ can be computed by (\ref{5.1}). Here
\begin{equation}\label{5.1}
\begin{split}
a_0^{(k)}
=&(4C_k+2^{k-1})a_0^{(1)}+\sum_{m=1}^{2\cdot2^{k-1}-1}{_0}C_m^ka_m^{(1)};\\
a_1^{(k)}
=&C_ka_0^{(1)}+\sum_{m=1}^{3\cdot2^{k-1}-1}{_1}C_m^k a_m^{(1)};\\
a_j^{(k)}
=&\sum_{m=(j-2)2^{k-1}}^{(j+2)2^{k-1}-1} {_j}C_m^k a_m^{(1)} \quad \forall j\geq 2 \quad \forall k\geq 2
    \end{split}
\end{equation}
with $C_k=2^{k-2}\cdot\frac{2^{2k-2}-1}{3}$. And
\begin{equation*}
{_0}C_m^k=\left\{ \begin{split}
&8C_k-(m^2-1)(2^k-m) ~~\quad~{\rm for}~~m=1: 2^{k-1};\\
&\frac{1}{3}(2^{k}-m-1)(2^{k}-m)(2^{k}-m+1)
~~{\rm for}~~m=2^{k-1}:2\cdot2^{k-1}-1;
 \end{split}
 \right.
\end{equation*}
${_1}C_m^k=$
\begin{equation*}
\left\{ \begin{split}
&2C_k+m^2\cdot2^{k-1}-\frac{2}{3}(m-1)m(m+1) ~~\quad~{\rm for}~~m=1: 2^{k-1};\\
&2C_k+(2^k-m)^2\cdot2^{k-1}-\frac{2}{3}(2^k-m-1)(2^k-m)(2^k-m+1)\\
& -\frac{1}{6}(m-2^{k-1}-1)(m-2^{k-1})(m-2^{k-1}+1)
~~\quad~{\rm for}~~~~m=2^{k-1}: 2\cdot2^{k-1};\\
&\frac{1}{6}(3\cdot2^{k-1}\!-\!m\!-\!1)(3\cdot2^{k-1}\!-\!m)(3\cdot2^{k-1}\!-\!m+1)
~~{\rm for}~~m=2\cdot2^{k-1}:3\cdot2^{k-1}-1;
 \end{split}
 \right.
\end{equation*}
and for $j\geq 2$,
\begin{equation*}
{_j}C_m^k=
\left\{ \begin{split}
&\varphi_1 ~~\quad~{\rm for}~~m=(j-2)2^{k-1}:(j-1)2^{k-1};\\
&\varphi_2
~~\quad~{\rm for}~~ m=(j-1)2^{k-1}: j2^{k-1};\\
&\varphi_3
~~\quad~{\rm for}~~ m=j2^{k-1}: (j+1)2^{k-1};
\\
&\varphi_4
~~\quad~{\rm for}~~ m=(j+1)2^{k-1}:(j+2)2^{k-1}-1,
 \end{split}
 \right.
\end{equation*}
where
\begin{equation*}
\begin{split}
\varphi_1=\frac{1}{6}(m-(j-2)2^{k-1}-1)(m-(j-2)2^{k-1})(m-(j-2)2^{k-1}+1);
 \end{split}
\end{equation*}
\begin{equation*}
\begin{split}
\varphi_2=&2C_k+(m-(j-1)2^{k-1})^2\cdot2^{k-1}\\
& -\frac{1}{6}(j2^{k-1}-m-1)(j2^{k-1}-m)(j2^{k-1}-m+1)\\
&-\frac{2}{3}(m-(j-1)2^{k-1}-1)(m-(j-1)2^{k-1})(m-(j-1)2^{k-1}+1);\\
 \end{split}
\end{equation*}
\begin{equation*}
\begin{split}
\varphi_3=&2C_k+((j+1)2^{k-1}-m)^2\cdot2^{k-1} \\
&-\frac{1}{6}(m-j2^{k-1}-1)(m-j2^{k-1})(m-j2^{k-1}+1)\\
&-\frac{2}{3}((j+1)2^{k-1}-m-1)((j+1)2^{k-1}-m)((j+1)2^{k-1}-m+1);\\
 \end{split}
\end{equation*}
\begin{equation*}
\begin{split}
\varphi_4=&\frac{1}{6}((j+2)2^{k-1}-m-1)((j+2)2^{k-1}-m)((j+2)2^{k-1}-m+1).
 \end{split}
\end{equation*}
\end{lemma}

\begin{lemma}\label{lemmma5.2}
Let $A^{(1)}=\{a_{i,j}^{(1)}\}_{i,j=1}^{\infty}$ with $a_{i,j}^{(1)}=a_{|i-j|}^{(1)}$ be a weakly diagonally dominant symmetric  Toeplitz M-matrix
and $D_{(k)}$ be the diagonal of the matrix $A^{(k)}$, where $A^{(k)}=I_k^{k-1}A^{(k-1)}I_{k-1}^{k}$. Then
$$1 \leq \lambda_{\max}\left(D_{(k)}^{-1}A^{(k)}\right) < 3.$$
In particular,
$$1 \leq \lambda_{\max}\left(D_{(k)}^{-1}A^{(k)}\right) \leq  2~~if ~~a_1^{(k)}\leq 0.$$
\end{lemma}
\begin{proof}
We take $A^{(k)}=L_h^{H}A^{(k-1)}L_{H}^{h}$ with $L_h^{H}=4I_k^{k-1}$ and $L_H^{h}=(L_h^{H})^T$,
and denote $$A^{(k)}=\{a_{i,j}^{(k)}\}_{i,j=1}^{\infty} ~~~~ {\rm with}~~~~ a_{i,j}^{(k)}=a_{|i-j|}^{(k)},~~~~\forall k \geq 1.$$
By mathematical induction, we prove the estimates
\begin{equation}\label{5.2}
  2\sum_{j=1}^{n}\left|a_j^{(s)}\right| \leq a_0^{(s)}~~~~{\rm if}~~a_1^{(s)}\leq 0,~~~~s\geq 2, n\geq 1;
\end{equation}
and
\begin{equation}\label{5.3}
2\sum_{j=1}^{n}\left|a_j^{(s)}\right| \leq \frac{2C_s+6C_s+2^{s-1}}{4C_s+2^{s-1}}a_0^{(s)}<2a_0^{(s)}~~~~{\rm if}~~a_1^{(s)}\geq 0.
\end{equation}

For $s=2$ and $a_1^{(2)}\leq 0.$  From  (\ref{5.1}) and (2.13) of \cite{Chen:16}, we get   $a_j^{(2)}\leq 0$, $j\geq 2$ and
\begin{equation*}
\begin{split}
2\sum_{j=1}^{n}\left|a_j^{(2)}\right|
&=-2\sum_{j=1}^{n}a_j^{(2)}\\
&=-2a_0^{(1)}+8\sum_{j=1}^{2n-1}|2a_j^{(1)}|+8a_1^{(1)}+2a_2^{(1)}+2a_{2n}^{(1)}+8a_{2n+1}^{(1)}+14a_{2n+2}^{(1)}\\
&\leq 6a_0^{(1)}+8a_1^{(1)}+2a_2^{(1)}=a_0^{(2)},
\end{split}
\end{equation*}
where we use the property of $A^{(1)}$, being a weakly diagonally dominant M-matrix.

For $s=2$ and $a_1^{(2)}\geq 0.$  According to  (\ref{5.1}) and (2.13) of \cite{Chen:16},  there exists  $a_j^{(2)}\leq 0$, $j\geq 2$ and
\begin{equation*}
\begin{split}
2\sum_{j=1}^{n}\left|a_j^{(2)}\right|
=&2a_1^{(2)}-2\sum_{j=2}^{n}a_j^{(2)}\\
=&2a_0^{(1)}+8\sum_{j=1}^{2n-1}|2a_j^{(1)}|+24a_1^{(1)}+26a_2^{(1)}+16a_3^{(1)} +4a_4^{(1)} \\
& +2a_{2n}^{(1)}+8a_{2n+1}^{(1)}+14a_{2n+2}^{(1)}  \\
\leq & \frac{10}{6}\left(6a_0^{(1)}+8a_1^{(1)}+2a_2^{(1)}\right)=\frac{5}{3}a_0^{(2)}.
\end{split}
\end{equation*}
Then (\ref{5.2}) and (\ref{5.3}) hold for $s=2$.
Suppose that (\ref{5.2}) and (\ref{5.3})  hold for $s=2,3,\ldots k-1$.
Next we prove (\ref{5.3})  holds for $s=k$.

Taking $a_1^{(k)}\geq 0$ and using (\ref{5.1}) and    $a_j^{(k)}\leq 0,~~j\geq 2$, we have
\begin{equation*}
\begin{split}
&2\sum_{j=1}^{n}\left|a_j^{(k)}\right|=2a_1^{(k)}-2\sum_{j=2}^{n}a_j^{(k)}\\
 =&2C_ka_0^{(1)}+2\sum_{m=1}^{2^{k-1}-1}({_1}C_m^k-{_2}C_m^k) a_m^{(1)}
+2\sum_{m=2^{k-1}}^{2\cdot2^{k-1}-1}({_1}C_m^k-{_2}C_m^k-{_3}C_m^k) a_m^{(1)}\\
&+2\sum_{m=2\cdot2^{k-1}}^{3\cdot2^{k-1}-1}({_1}C_m^k-{_2}C_m^k-{_3}C_m^k-{_4}C_m^k) a_m^{(1)}\\
&-2\sum_{j=2}^{n-3}\sum_{m=(j+1)2^{k-1}}^{(j+2)2^{k-1}-1}({_j}C_m^k+{_{j+1}}C_m^k+{_{j+2}}C_m^k+{_{j+3}}C_m^k) a_m^{(1)} \\
&-2\sum_{m=(n-1)2^{k-1}}^{n2^{k-1}-1}({_{n-2}}C_m^k+{_{n-1}}C_m^k+{_{n}}C_m^k) a_m^{(1)}\\
&-2\sum_{m=n2^{k-1}}^{(n+1)2^{k-1}-1}({_{n-1}}C_m^k+{_{n}}C_m^k) a_m^{(1)}-2\sum_{m=(n+1)2^{k-1}}^{(n+2)2^{k-1}-1}{_{n}}C_m^k a_m^{(1)}\\
\leq & \frac{2C_k+6C_k+2^{k-1}}{4C_k+2^{k-1}}a_0^{(k)}<2a_0^{(k)},
\end{split}
\end{equation*}
where we use ${_j}C_m^k+{_{j+1}}C_m^k+{_{j+2}}C_m^k+{_{j+3}}C_m^k=6C_k+2^{k-1}$
 and ${_1}C_m^k-{_2}C_m^k+6C_k+2^{k-1} \geq {_{0}}C_m^k,~m=1: 2^{k-1}$
 and ${_1}C_m^k-{_2}C_m^k-{_3}C_m^k +6C_k+2^{k-1} \geq {_{0}}C_m^k,~~~m=2^{k-1}: 2\cdot2^{k-1}$
 and the property of $A^{(1)}$.
Similarly, we can prove (\ref{5.2}).
Then we obtain
\begin{equation*}
\begin{split}
1 \leq \lambda_{\max}\left(D_{(k)}^{-1}A^{(k)}\right) < 3,~~a_1^{(k)}\geq 0
~~{\rm { and}}~~
1 \leq \lambda_{\max}\left(D_{(k)}^{-1}A^{(k)}\right) \leq  2,~~a_1^{(k)}\leq 0.
  \end{split}
\end{equation*}
The proof is completed.
\end{proof}

\subsection{The operation count and storage requirement}\label{section:3.2}
We now discuss the computation count and the required storage for the MGM of nonlocal problem (\ref{2.10}).

From (\ref{2.11}), we know that the  matrix $A_h=A_{\delta}^h$ is
a symmetric banded Toeplitz matrix with bandwidth $r+1$ (obviously less than $N$).
Then, we only need to store the first column of $A_h$,
which have $N$ parameters, instead of the full matrix $A_h$ with $N^2$ entries.
From Lemma \ref{lemmma5.2}, we know that $\{A_k\}$ is the symmetric Toeplitz matrix  with the grid sizes $\{2^{J-k}h\}_{k=1}^{J-1}$,
i.e., $\mathcal{M}_{k}$ requires $2^{J-k}N$ storage.
Adding these terms together, we
\begin{equation*}
  \mbox{Storage}=\mathcal{O}(N) \cdot \left( 1+\frac{1}{2}+\frac{1}{2^2}+\ldots,+\frac{1}{2^{J-1}} \right)=\mathcal{O}(N).
\end{equation*}

As for operation counts, the matrix-vector product associated with the matrix $A_h$ is a discrete convolution.
While the cost of a direct product is $O(r N)$, the cost of using the FFT would lead to $O(N\log(N)$ \cite{Chan:07}.
Moreover, from (\ref{5.1}), we know that the bandwidth of $\{A_k\}$  is not bigger than  the bandwidth of  $A_h$.
Hence, with the change of $r$, we may adopt different strategies. Thus, the total
per V-cycle MGM operation count is
\begin{equation*}
\mathcal{O}\left(  \min\{ rN,N\log N\} \right)
\cdot \left( 1+\frac{1}{2}+\ldots,+\frac{1}{2^{J-1}} \right)
\!= \!\left\{
 \begin{array}{ll}\!\!
  \mathcal{O}(N), & r~\mbox{is bounded}, \\
\!\!\mathcal{O}(N \log (N),   \!\! & \mbox{in the worst case}.
  \end{array}\right.
\end{equation*}
\subsection{Convergence of the full MGM with $\bf{\delta=ch}$}
We further consider the convergence of the full MGM (recursive application of the TGM procedure).
More precisely we show that the constants $\eta$ in Lemma \ref{lemma4.7} and $\kappa$ in Lemma \ref{lemma4.8} do not depend on the levels; this level independence is crucial
for the convergence theory of the full MGM \cite{Chan:98,Serra:02}. In the following, we consider the simple algebraic systems, but
these algebraic arguments are mostly motivated from analytic considerations.
\begin{lemma}\label{lemma}\label{lemma5.3}
Let  $A^{(1)}=c_1L_1 + c_2L_2+c_3L_3$, $c_i>0$,  $i=1,2,3$ and $A^{(k)}=I_k^{k-1}A^{(k-1)}I_{k-1}^{k}$.  Then
\begin{equation*}
||K_{k} T^k||_{A_k} \leq \sqrt{1-\eta/\kappa }<1~~\forall ~1\leq k \leq J,
\end{equation*}
where $A_k=A^{(J-k+1)}$, $\kappa=\max \left\{ 1+\frac{c_2+c_3}{c_1},~3    \right \}$, and  $\eta\leq2\omega(1-\omega)$ with $0 < \omega <1$.
\end{lemma}

\begin{proof}
Let $A^{(k)}=L_h^{H}A^{(k-1)}L_{H}^{h}$ with $L_h^{H}=4I_k^{k-1}$ and $L_H^{h}=(L_h^{H})^T$, and $A^{(k)}=\{a_{i,j}^{(k)}\}_{i,j=1}^{\infty}$ with $a_{i,j}^{(k)}=a_{|i-j|}^{(k)}$,
where
\begin{equation}\label{5.4}
\begin{split}
A^{(1)}
=&c_1 \cdot{\rm diag}\left(-1, 2,-1\right)
+c_2\cdot{\rm diag}\left(-1,0, 2,0, -1\right)\\
&+c_3\cdot{\rm diag}\left(-1,0,0, 2,0, 0,-1\right).
\end{split}
\end{equation}
According to Lemma \ref{lemma5.1} and above equations, we obtain
\begin{equation}\label{5.5}
\begin{split}
A^{(k)}
=&c_1 \cdot{\rm diag}\left(-\frac{d_1^{(k)}}{2}, d_1^{(k)},-\frac{d_1^{(k)}}{2}\right)\\
&+c_2\cdot{\rm diag}\left(-1,-\frac{d_2^{(k)}-2}{2}, d_2^{(k)},-\frac{d_2^{(k)}-2}{2}, -1\right)\\
&+c_3\cdot{\rm diag}\left(-4,-\frac{d_3^{(k)}-8}{2}, d_3^{(k)},-\frac{d_3^{(k)}-8}{2}, -4\right)\\
=&\sigma_1^{(k)}L_1+\sigma_2^{(k)}L_2,
\end{split}
\end{equation}
where
\begin{equation*}
\begin{split}
\sigma_1^{(k)}=c_1 \frac{d_1^{(k)}}{2}+c_2\frac{d_2^{(k)}-2}{2}+c_3\frac{d_3^{(k)}-8}{2},~~~~\sigma_2^{(k)}=c_2+4c_3,~~~~ 2\leq k\leq J
\end{split}
\end{equation*}
with $d_1^{(k)}=2^{k}$, $d_2^{(k)}=2^{k+2}-6$  and $d_3^{(k)}=9\cdot2^{k}-24$.

Using (\ref{4.16}) and (\ref{5.4}), there exists
\begin{equation*}
\begin{split}
  \kappa
  &=\frac{a_0^{(k)}}{-2a_1^{(k)}}=\frac{2\left( \sigma_1^{(k)}+\sigma_2^{(k)} \right)}{2\sigma_1^{(k)}}=1+\frac{c_2+4c_3}{\sigma_1^{(k)}}\\
  &\leq 1+\frac{c_2+4c_3}{\sigma_1^{(2)}}=1+\frac{c_2+4c_3}{2c_1+4c_2+2c_3}<3,~~~~ 2\leq k\leq J;
\end{split}
\end{equation*}
and
\begin{equation*}
\begin{split}
  \kappa
  &=\frac{a_0^{(1)}}{-2a_1^{(1)}}=\frac{2\left( c_1+c_2+c_3 \right)}{2c_1}=1+\frac{c_2+c_3}{c_1},~~~~ k= 1.
\end{split}
\end{equation*}
Combining the proof of Theorem \ref{theorema4.9}, the desired results are obtained.
\end{proof}

\begin{theorem}
Let $A_J=A^h_\delta$ be defined by  (\ref{2.10}) and (\ref{2.15}) on a finite bar $\Omega \in(0,b)$, where  $\delta=Rh$,  $R=3$, $h\rightarrow 0$.
Then
\begin{equation*}
||K_{k} T^k||_{A_k} \leq \sqrt{1-\eta/6 }<1,~~~~1\leq k \leq J,
\end{equation*}
where $\eta\leq2\omega(1-\omega)$ with $0 < \omega <1$.
\end{theorem}
\begin{proof}
According to Lemma \ref{lemma5.3} and Theorem \ref{theorema4.11}, there exists

\begin{equation*}
||K_{k} T^k||_{A_k} \leq \sqrt{1-\eta/3 }<1,~~~~1\leq k \leq J-1,
\end{equation*}
and
\begin{equation*}
||K_{J} T^J||_{A_J} \leq \sqrt{1-\eta/6 }<1,~~~~k =J.
\end{equation*}
The proof is completed.
\end{proof}
\subsection{Convergence of the V-cycle MGM with $\bf{\delta=h}$}
In the special case, the convergence of the V-cycle MGM can also be simply obtained. Firstly, we have the following lemma.
\begin{lemma}[\cite{Chen:16}]\label{lemma5.9}
Let the symmetric positive definite matrix $A_k$ satisfy
\begin{equation}\label{5.6}
\frac{\omega}{\lambda_{\max}(A_k) }(\nu^k,\nu^k) \leq (S_k\nu^k,\nu^k)\leq (A_k^{-1}\nu^k,\nu^k)~~~\forall \nu^k\in \mathcal{B}_k,
\end{equation}
and
\begin{equation}\label{5.7}
   \min_{\nu^{k-1} \in \mathcal{B}_{k-1} }||\nu^k-I_{k-1}^k\nu^{k-1}||_{A_k}^2\leq m_0 ||A_k\nu^k||_{D_k^{-1}}^2 \quad  \forall \nu^k \in \mathcal{B}_k
\end{equation}
with  $m_0>0$ independent of $\nu^k$. Then
$$||I-B_kA_k||_{A_k} \leq \frac{m_0}{2l\omega+m_0}<1~~{\rm with}~~~1\leq k\leq K,$$
where  the operator $B_k$ is defined by the  V-cycle method in   Algorithm  \ref{MGM}  and  $l$ is the number of smoothing steps.
\end{lemma}

We know that $A^h_\delta$ reduces to the second order elliptic operator when $R\leq 1$ in  (\ref{2.10}).
From \cite{Chen:16,Xu:92,Xu:02,Xu:97}, it is easy to check that  (\ref{5.6}) and (\ref{5.7}) hold with $m_0=1$.
Then we have the following  results.
\begin{theorem}\label{theorem5.10}
Let $A_J=A^h_\delta$ be defined by  (\ref{2.10}) and (\ref{2.15}) on a finite bar $\Omega \in(0,b)$, where  $\delta=Rh$,  $R\leq 1$, $h\rightarrow 0$.
Then
$$||I-B_kA_k||_{A_k} \leq \frac{1}{2l\omega+1}<1~~{\rm with}~~~1\leq k\leq J, ~~~~\omega \in (0,1/2],$$
where  the operator $B_k$ is defined by the  V-cycle method in  Multigrid Algorithm  \ref{MGM}  and  $l$ is the number of smoothing steps.
\end{theorem}
\begin{remark}
Using  Lemma \ref{lemmma5.2}, we know that (\ref{5.6}) holds for the general nonlocal models or fractional models  \cite{Chen:14,Pang:12} with $\omega \in (0,1/3]$,
but it is not easy to check the condition (\ref{5.7}).
\end{remark}
\section{Numerical Results}
We employ the V-cycle MGM  described in Algorithm  \ref{MGM} to solve the steady-state nonlocal problem (\ref{2.3}).
The stopping criterion is taken as
$\frac{||r^{(i)||}}{||r^{(0)}||}<10^{-8},$
where $r^{(i)}$ is the residual vector after $i$ iterations;
and the  number of  iterations $(m_1,m_2)=(1,2)$ and $(\omega_{pre},\omega_{post})=(1,1/3).$
In all tables, $N$  denotes the number of spatial grid points; and the numerical errors are measured by the $ l_{\infty}$
(maximum) norm,  `Rate' denotes the convergence orders.
`CPU' denotes the total CPU time in seconds (s) for solving the resulting discretized  systems;
and `Iter' denotes the average number of iterations required to solve a general linear system $A_hu_h=f_h$ at each time level.

All numerical experiments are programmed in Matlab, and the computations are carried out  on a laptop with the configuration: Inter(R) Core (tm) i3  CPU 2.27 GHZ and 2 GB RAM and a Windows 7 operating system.
\begin{example}
Consider the steady-state nonlocal problem
$$- \mathcal{L}_\delta u (x)=-12x^2+12bx-2b^2-\frac{6}{5}\delta^2$$
with a finite domain $0< x < b $, $b=4$. The exact solution of the equation is $u(x)=x^2(b-x)^2$, and the boundary conditions  $u=g$ on $\Omega_\mathcal{I}$.
\end{example}

\begin{table}[h]\fontsize{9.5pt}{12pt}\selectfont
  \begin{center}
  \caption{Using {\bf Galerkin approach}  $A_{k-1}=I_k^{k-1}A_kI_{k-1}^{k}$ computed by (\ref{4.6}) to solve the
  resulting systems (\ref{2.10})  with $h=4/N$.}\vspace{5pt}
 {\small   \begin{tabular*}{\linewidth}{@{\extracolsep{\fill}}*{9}{c}}                                    \hline  
$N$            &  $\delta=1$   & Rate    & Iter    & CPU       &  $\delta=\sqrt{h}$  &   Rate &  Iter  &  CPU    \\\hline
    $2^{10}$   &   4.0638e-05  &         &  13     & 0.043 s   &  3.1010e-05         &        &  21    & 0.067 s\\\hline 
    $2^{11}$   &   1.0169e-05  & 2.00    &  13     & 0.067 s   &  7.7246e-06         &  2.01  &  21    & 0.103 s \\\hline 
    $2^{12}$   &   2.5461e-06  & 2.00    &  12     & 0.101 s   &  1.9262e-06         &  2.00  &  21    & 0.175 s \\\hline 
    $2^{13}$   &   6.3918e-07  & 2.00    &  12     & 0.203 s   &  4.8200e-07         &  2.00  &  21    & 0.308 s \\\hline 
$N$            &  $\delta=5h$  & Rate    & Iter    & CPU       &  $\delta=h$         &   Rate &  Iter  &  CPU    \\\hline
    $2^{10}$   &   3.0396e-05  &         &  22     & 0.069 s   &  2.4416e-05         &        &  18    & 0.058 s\\\hline 
    $2^{11}$   &   7.5840e-06  & 2.00    &  23     & 0.112 s   &  6.1057e-06         &  2.00  &  18    & 0.091 s \\\hline 
    $2^{12}$   &   1.8943e-06  & 2.00    &  23     & 0.196 s   &  1.5310e-06         &  2.00  &  18    & 0.146 s \\\hline 
    $2^{13}$   &   4.7244e-07  & 2.00    &  23     & 0.384 s   &  3.8268e-07         &  2.00  &  18    & 0.261 s \\\hline 
    \end{tabular*}}\label{tab:1}
  \end{center}
\end{table}

\begin{table}[h]\fontsize{9.5pt}{12pt}\selectfont
  \begin{center}
  \caption{Using {\bf doubling the mesh size} $A_{k-1}=A_{\delta}^{2^{K-k+1}h} $ to solve the
  resulting systems (\ref{2.10})  with $h=4/N.$}\vspace{5pt}
 {\small   \begin{tabular*}{\linewidth}{@{\extracolsep{\fill}}*{9}{c}}                                    \hline  
$N$            &  $\delta=1$   & Rate    & Iter    & CPU       &  $\delta=\sqrt{h}$  &  Rate  &  Iter  &  CPU    \\\hline
    $2^{10}$   &   4.0589e-05  &         &  42     & 0.133 s   &  3.0999e-05         &        &  56    & 0.175 s\\\hline 
    $2^{11}$   &   1.0132e-05  & 2.00    &  40     & 0.195 s   &   7.7139e-06        &  2.01  &  54    & 0.259 s \\\hline 
    $2^{12}$   &   2.5229e-06  & 2.01    &  39     & 0.312 s   &  1.9184e-06         &  2.01  &  54    & 0.430 s \\\hline 
    $2^{13}$   &   6.2373e-07  & 2.02    &  38     & 0.542 s   &  4.7520e-07         &  2.01  &  53    & 0.757 s \\\hline 
$N$            &  $\delta=5h$  & Rate    & Iter    & CPU       &  $\delta=h$         &  Rate  &  Iter  &  CPU    \\\hline
    $2^{10}$   &   3.0393e-05  &         &  54     & 0.169 s   &  2.4385e-05         &        &  47    & 0.160 s\\\hline 
    $2^{11}$   &   7.5819e-06  & 2.00    &  54     & 0.261 s   &  6.0749e-06         &  2.01  &  47    & 0.242 s \\\hline 
    $2^{12}$   &   1.8917e-06  & 2.00    &  53     & 0.422 s   &  1.4982e-06         &  2.02  &  47    & 0.372 s \\\hline 
    $2^{13}$   &   4.7021e-07  & 2.01    &  52     & 0.752 s   &  3.7118e-07         &  2.01  &  47    & 0.665 s \\\hline 
    \end{tabular*}}\label{tab:2}
  \end{center}
\end{table}

We use two coarsening strategies: {\em Galerkin approach} and {\em doubling the mesh size}, respectively, to solve the
resulting system (\ref{2.10}).  Tables \ref{tab:1} and \ref{tab:2} show that these two methods have almost the same error values
and the numerically  confirm that the numerical scheme has second-order accuracy and the computation cost is almost  $\mathcal{O}(N \mbox{log} N)$ operations.

\section{Conclusions}

There are already some theoretical convergence results for using the multigrid method to solve the PDEs, the algebraic system of which has the Toeplitz structure. We notice that the proofs are mainly based on the boundedness of $a_0/a_1$, where $a_0$ and $a_1$ are, respectively, the principal diagonal element and the trailing diagonal element of the Toeplitz matrix. However, in the nonlocal system, most of the time the boundedness of $a_0/a_1$ does not hold again. In this work, we rewrite the corresponding symmetric Toeplitz matrix as a sum of a series of Laplacian-like matrices. Then based on the analysis of the Laplacian-like matrix, we present the strict proof of the uniform convergence of the TGM. And the convergence results of the full MGM and V-cycle MGM in a special case are also derived.  For the framework of the uniform convergence of the V-cycle MGM, the condition (\ref{5.6}) has been confirmed to hold for the class of  weakly diagonally dominant symmetric  Toeplitz  M-matrices, in the future we will try to find the way to verify the condition (\ref{5.7}).

\section*{Acknowledgments}
The first author wishes to thank Qiang Du  for his valuable comments while working in Columbia university.

\bibliographystyle{amsplain}

\end{document}